\documentclass{amsart}

\usepackage{amsmath}
\usepackage{amsfonts}
\usepackage{amssymb,enumerate}
\usepackage{amsthm}
\usepackage{cite}
\usepackage{comment}
\usepackage{color}
\usepackage[all]{xy}
\usepackage{hyperref}

\theoremstyle{definition}
\newtheorem{lem}{Lemma}[section]
\newtheorem{cor}[lem]{Corollary}
\newtheorem{prop}[lem]{Proposition}
\newtheorem{thm}[lem]{Theorem}

\newtheorem{defn}[lem]{Definition}
\newtheorem{ex}[lem]{Example}

\newtheorem{notation}[lem]{Notation}

\newtheorem{rem}[lem]{Remark}

\renewcommand{\geq}{\geqslant}
\renewcommand{\leq}{\leqslant}

\numberwithin{equation}{section}

\begin{document}
\author{Jim Coykendall}
\address{School of Mathematical and Statistical Sciences\\
	Clemson University\\
	Clemson, SC 29634}
\email[J.~Coykendall]{jcoyken@clemson.edu}
\author{Tridib Dutta}
\email[T.~Dutta]{dutta.tridib@gmail.com}

\title{Rings of Very Strong Finite Type}
\begin{abstract}
The SFT (for {\it strong finite type}) condition was introduced by J. Arnold \cite{Ar1973} in the context of studying the condition for formal power series rings to have finite Krull dimension. In the context of commutative rings, the SFT property is a near-Noetherian property that is necessary for a ring of formal power series to have finite Krull dimension behavior. Many others have studied this condition in the context of the dimension of formal power series rings. In this paper, we explore a specialization (and in some sense a more natural) variant of the SFT property that we dub the VSFT (for {\it very strong finite type}) property. As is true of the SFT property, the VSFT property is a property of an ideal that may be extended to a global property of a commutative ring with identity. Any ideal (resp. ring) that has the VSFT property has the SFT property. In this paper we explore the interplay of the SFT property and the VSFT property.
\end{abstract}

\maketitle

\section{Introduction and Background}

In this paper, we introduce the notion of {\it very
strong finite type} (VSFT) to the study of commutative rings and their ideals. This property is a strengthening of the strong
finite type (SFT) property, which was first introduced by Arnold (see \cite{Ar1973}) in the context of studying the dimension of formal power series rings. It is well known that if $R$ is a commutative Noetherian ring with identity and $\text{dim}(R)=n$ then $\text{dim}(R[[x]])=n+1$. In the non-Noetherian case, precious little is known. The first major result was the following theorem from \cite{Ar1973}:

\begin{thm}
Let $R$ be a commutative ring with identity. If $R$ does not have the SFT property then $\text{dim}(R[[x]])$ is infinite.
\end{thm}

An obvious question motivated by the previous theorem is, ``if $R$ is an SFT ring of finite dimension, is it true that $\text{dim}(R[[x]])$ is finite?". This question was answered in the negative in \cite{Co2002} and this work has been augmented in a number of works including \cite{kp2} and \cite{KLLP}. The SFT property has been studied as a tool in studying Krull dimension (especially in power series rings) as in \cite{TK}, \cite{CKT}, \cite{KLLP} and also in studying more general chains of primes and the behavior of the prime spectrum as in \cite{LL} and \cite{LL2}. Additionally, the SFT property and some of its generalizations have been studied in their own right in (for example) \cite{coco}, \cite{ek}, and \cite{kp}.

As a property of a commutative ring, the SFT property is a near-Noetherian property, and hence there are a number of natural questions that arise, some of which have been answered and some have not. For example, it is known (see \cite{Co2002}) that if $R$ is SFT, then $R[[x]]$ may not be (in contrast with the well known fact that if $R$ is Noetherian, then $R[[x]]$ is Noetherian). On the other hand, if $R$ is SFT, the status of $R[x]$ (that is, whether the SFT property persists in the polynomial extension) remains open.

The aim of this paper is to look at the SFT and VSFT (to be defined presently) properties and explore their interrelation and compare and contrast with the Noetherian property.

We begin by recalling the notion of SFT ideals and the notion of an SFT ring. In this paper, ``ring" refers to a commutative ring with identity.

\begin{defn}\label{SFTdef}
Let $R$ be a ring and $I$ be an ideal of $R$. Then $I$
is called an SFT ideal if there exists a finitely generated ideal
$\mathcal{B} \subseteq I$ and a positive integer $n$ such that $x^n
\in \mathcal{B}$ for all $ x \in I$.
\end{defn}

\begin{defn}
If every ideal of a ring, $R$, is SFT, then we say that $R$ is an SFT ring.
\end{defn}

The following notation will streamline our work; it will appear extensively in the sequel.

\begin{notation}
Suppose as in Definition \ref{SFTdef} we have an SFT ideal $I$, a finitely generated ideal $\mathcal{B}\subseteq I$, and a positive integer $n$ such that $x^n\in\mathcal{B}$ for all $x\in I$, we will say that $I$ is SFT with data $(I,\mathcal{B},n)$, and we refer to $n$ as the index. We note that the choices of $\mathcal{B}$ and $n$ are not unique in general.
\end{notation}

\section{Motivating Examples, the VSFT Property, and Ideal-Theoretic Results}

It is natural to wonder that if $I$ is an SFT ideal with data $(I,\mathcal{B},n)$, then is there a positive integer $m$ such that $I^m\subseteq \mathcal{B}$? In the special case where the SFT ideal is finitely generated, this is certainly the case (and so rings with this property would be a natural generalization of Noetherian rings). However, this property does not persist in general as we see in the following example.

\begin{ex}
\label{SFTnotVSFTex1}
Let $\mathbb{F}_p$ be a field of characteristic $p$; we begin with the domain $R=\mathbb{F}_p[x_1,x_2,\cdots, x_n,\cdots]$ and let $I\subseteq R$ be the ideal $I:=(x_1^p,x_2^p,\cdots)$. We let the image of $x_i$ in $R/I$ be denoted by $\overline{x}_i$, and note that $T:=R/I$ is a zero-dimensional ring with unique prime ideal $\mathfrak{M}=(\overline{x}_1,\overline{x}_2,\cdots)$. Note that each $z\in\mathfrak{M}$ has the property that $z^p=0$ and hence $\mathfrak{M}$ is SFT with data $(\mathfrak{M},0,p)$. However there is no positive integer $n$ and finitely generated ideal $\mathcal{B}$ such that $\mathfrak{M}^n\subseteq\mathcal{B}$. To see this, we suppose that there is a finitely generated $\mathcal{B}$ and positive integer $n$ such that $\mathfrak{M}^n\subseteq\mathcal{B}$. Since $\mathcal{B}$ is finitely generated (say by $y_1,y_2,\cdots, y_t$) and each generator $y_i$ has the property that $y_i^p=0$, it is easy to see that $\mathcal{B}$, and hence $\mathfrak{M}$, is a nilpotent ideal. But for all $n$, the element $\overline{x}_{1}\overline{x}_{2}\cdots \overline{x}_{n}$ is a nonzero element of $\mathfrak{M}^n$, contradicting its nilpotency. 
\end{ex}

In the above example, the computation depends heavily on the fact
that the ring is of characteristic $p$. If we try to mimic this
example in the characteristic 0 situation, we encounter some obvious
difficulties. In
fact, it turns out that a similar example in the zero characteristic
situation is more complicated in nature. We shall provide such an
example in due course. But at this moment, we will produce several
examples laying the groundwork for further study.

Note that in Example \ref{SFTnotVSFTex1}, we used a non-domain to illustrate the point we were trying to make. We will now provide an example of an SFT ideal that is not finitely generated in an integral domain setting.

\begin{ex}\label{SFTnotVSFTex2}

Let us consider the following domain 

\[
R =
{\mathbb{F}_2[y,x_1,x_2,\cdots,\frac{y}{x_1},\frac{y}{x_1^2},\cdots,\frac{y}{x_2},\frac{y}{x_2^2},\cdots]}
\]

\noindent Now consider the ideal $I =
(y,\frac{y}{x_1},\frac{y}{x_1^2},\cdots,\frac{y}{x_2},\frac{y}{x_2^2},\cdots)$.
Clearly, the generators of the ideal satisfy the  SFT property
since $(\frac{y}{x_k^m})^2 = (y)(\frac{y}{x_k^{2m}})\in (y), $ for
all $ m > 0$. Since the characteristic of the ring is 2, it is not
difficult to verify that any arbitrary element of $I$ raised to a
power of 2 belongs to the ideal $(y)$. This shows that $I$ is SFT.
\end{ex}

In Example \ref{SFTnotVSFTex2}, the product of two arbitrary generators of $I$ do not belong to $(y)$ (for instance $(\frac{y}{x_1})(\frac{y}{x_2})$), that is, $I^2 \nsubseteq (y)$.

The previous two examples share the theme that although $I$ is an SFT ideal, no power of $I$ can be put inside the finitely generated sub-ideal. This motivates the following definition of a specialized version of SFT ideals that is perhaps more ideal-theoretic (as opposed to the more elemental nature of the SFT concept). 

\begin{defn}\label{VSFTdef}
Let $R$ be a ring and $I\subseteq R$ an ideal. Then $I$
is said to be of very strong finite type (VSFT) if there exists a
finitely generated ideal $\mathcal{B}$ and a positive integer $n$
such that $I^n \subseteq \mathcal{B}$. Similar to the SFT case, we will say such a VSFT ideal is VSFT with data $(I,\mathcal{B},n)$.
\end{defn}

\begin{defn}
If every ideal of a ring, $R$, is VSFT,  then we say that the $R$ is a VSFT ring.
\end{defn}

We remark that it is clear that any ideal that has the VSFT property is SFT (and hence any VSFT ring is SFT).

We make a couple of easy, but useful observations.

\begin{rem}
If $I$ is a VSFT ideal with data $(I,\mathcal{B}, N)$ then $I$ is VSFT with data $(I,\mathcal{B},k)$ for every integer $k \geq N$.
\end{rem}

\begin{rem}
If $I$ is finitely generated then it is trivially
VSFT (the VSFT data, in this case, can be given by $\{I, I, 1\}$). So in
particular, a Noetherian ring is also a VSFT ring.
\end{rem}


Since Example \ref{SFTnotVSFTex1} and Example \ref{SFTnotVSFTex2}, are SFT ideals which are not VSFT, we now produce a nontrivial example of a VSFT ideal.

\begin{ex}\label{VSFTex1}

Consider the ideal $I = (2,
2x,2x^2,2x^3,\cdots,2x^n,\cdots)\subseteq \mathbb{Z}+2x\mathbb{Z}[x]$. As was noted earlier $x \notin R$
and so the above ideal is not finitely generated. Consider the ideal
generated by the element 2. Then $(2)\subset I$. The product of any
two distinct generators of the ideal is in $(2)$, that is,
$(2x^m)(2x^k)= (2)(2x^{m+k}) \in (2)$ as $2x^{m+k}\in R$. More
generally, let us consider two elements from the ideal $I$. They are
$\alpha = 2x^{k_1}p_1(x)+\cdots +2x^{k_r}p_r(x)$, where $p_i(x) \in R$
and $\gamma = 2x^{m_1}q_1(x)+\cdots +2x^{m_t}q_t(x)$, where
 $q_j(x) \in R $. Then it is clear that $\alpha \gamma \in (2)$. Hence $I$ is a VSFT
 ideal.
\end{ex}

A rich source of nontrivial VSFT ideals can be constructed from almost integral elements. The next result highlights this connection and generalizes the previous example.

\begin{thm}\label{almostIntegral}
Let $K$ be the field of fractions of the domain $R$. Let $\alpha \in K$ be an almost integral element over $R$ and $0\neq r\in R$ an element such that $r\alpha^n\in R$ for all $n\in\mathbb{N}$. Let $\mathfrak{M} = (\alpha, \alpha^2,\cdots)$ be the $R-$submodule of $K$ generated by $\alpha$ and all its powers. Then the ideal $(r,r\mathfrak{M}) \subseteq R$ is VSFT with data $\{(r,r\mathfrak{M}),(r),2\}$.
\end{thm}

\begin{proof}
If $rm_1,rm_2$ are such that $m_1,m_2\in\mathfrak{M}$, then $rm_1m_2\in r\mathfrak{M}\subset R$.
\end{proof}

A direct consequence of the definition of VSFT (resp. SFT) ideals is captured in the following proposition.

\begin{prop}\label{equalRadical}
Let $R$ be a ring and $I$ be a VSFT (resp. SFT) ideal with data $\{I,\mathcal{B},n\}$. Then $\sqrt{I} = \sqrt{\mathcal{B}}$.
\end{prop}

\begin{proof}
By the definition of an SFT ideal, we have $\mathcal{B} \subseteq I \subseteq \sqrt{I} \subseteq \sqrt{\mathcal{B}}$ and hence equality is immediate. Since any VSFT ideal is SFT, the proof is complete.
\end{proof}

As a corollary, we have the following result.

\begin{cor}
If $Q$ is a radical VSFT (resp. SFT) ideal in $R$ 
with data $\{Q, \mathcal{B}, n\}$, then $Q = \sqrt{\mathcal{B}}$.
\end{cor}

What follows from Proposition \ref{equalRadical} is the fact that if $I$ is not finitely generated, the sub-ideal $\mathcal{B}$ cannot be a radical ideal. We capture this observation in the following proposition.


\begin{prop}
Let $I$ be a VSFT ideal with data $(I, \mathcal{B}, N)$, where $I$ is not finitely generated. Then
$\mathcal{B}$ cannot be a radical ideal (and hence cannot be prime).
\end{prop}

With this basic groundwork laid and some initial examples presented, we now compare and contrast some of the basic properties of the SFT and VSFT properties.


We begin this section by observing the fact that if $I$ is a VSFT (resp. SFT) ideal, then any power of $I$ is also a VSFT (resp. SFT) ideal (with potentially varying (V)SFT data).

\begin{prop}\label{powerVSFT}
Let $R$ be a ring and $I$ be a VSFT (resp. SFT) ideal with VSFT data (resp. SFT data ) $(I,\mathcal{B},n)$ . Then $I^m$ is also a VSFT ideal (resp. SFT) ideal with VSFT data (resp. SFT data) $(I^m,\mathcal{B}^m,n)$ (resp. $(I^m,\mathcal{B}^m,mn)$) for all integers $m \geq 1$. 
\end{prop}

\begin{proof}

For the VSFT case, we note that $I^n\subseteq\mathcal{B}$ and so $(I^n)^m=(I^m)^n\subseteq\mathcal{B}^m$. Since $\mathcal{B}^m$ is finitely generated, $I^m$ is VSFT with data $(I^m,\mathcal{B}^m,n)$.

In the case that $I$ is SFT with data $(I,\mathcal{B}, n)$, we first note that since $\mathcal{B}\subseteq I$, then $\mathcal{B}^{mn}\subseteq\mathcal{B}^m\subseteq I^m$ and is finitely generated. So, if $\alpha\in I^m\subseteq I$, then $\alpha^n\in\mathcal{B}$ and hence $\alpha^{mn}\in\mathcal{B}^m$. This establishes the proposition. 
\end{proof}

Given an (V)SFT ideal with data $\{I.\mathcal{B},n\}$, we have established that $I^m$ is also (V)SFT with related data. It would be interesting to know what is the smallest positive integer $k$ (in terms of $n$ and $m$) such that there is a finitely generated ideal $J$ with the property that $I^m$ has (V)SFT data $\{I^m, J, k\}$.

The next result in this section is a central result which shows that the (V)SFT property is determined by the prime spectrum; that is, a ring is (V)SFT if and only if each prime ideal is (V)SFT.

\begin{thm}\label{PrimeSufficient}
A ring $R$ is VSFT (resp. SFT) if and only if every prime ideal $P\subseteq R$ is VSFT (resp. SFT).
\end{thm}

\begin{proof}
The statement for the SFT case was established in \cite{Ar1973} and so we will only consider the VSFT case.

We first note that if $R$ is VSFT, then every ideal of $R$, in particular, every prime ideal of $R$, is VSFT.

For the other direction, we assume that every prime ideal of $R$ is VSFT, but $R$ itself is not VSFT. By assumption, we can find an ideal $I\subseteq R$ which is not a VSFT ideal. Let $\Gamma$ be the set of
all non-VSFT ideals of $R$ containing the ideal $I$. $\Gamma$ is a non-empty, partially ordered set under set-theoretic inclusion. If we let $\mathfrak{C}$ be a chain in
$\Gamma$ and $\mathcal{P}= \bigcup_{J\in\mathfrak{C}} J$. then $\mathcal{P} $ is an upper bound for $\mathfrak{C}$ provided that $\mathcal{P}$ is not VSFT.

We first establish that $\mathcal{P}$ is not a VSFT ideal. To this end, we assume that $\mathcal{P}$ is VSFT with data $\{\mathcal{P},\mathcal{B}, N\}$. Additionally, we assume that $\mathcal{B}=(y_1,y_2,y_3,\cdots,y_k)$ 
Since $\mathcal{P}$ is a union of the chain of ideals, $\mathfrak{C}$, each $y_i\in J_{\alpha_i}$ for some $J_{\alpha_i}\in\mathfrak{C}$. Since $\mathfrak{C}$ is a chain, we can say, without loss of generality, that each $J_{\alpha_i}\subseteq J_{\alpha_k}$ for all $1\leq i\leq k$. Hence, $\mathcal{B}\subseteq J_{\alpha_k}\subseteq\mathcal{P}$. By assumption, we have that $\mathcal{P}^N\subseteq\mathcal{B}$, and hence $J_{\alpha_k}^N\subseteq\mathcal{B}$. This contradicts the fact that $J_{\alpha_k}$ is not VSFT and establishes our first claim.

Our next step will be to show that $\mathcal{P}$ is a prime ideal. To this end, we assume that $\mathcal{P}$ is not prime and find two elements $a,b\notin R$ such that $ab \in \mathcal{P}$. Then the ideals $(\mathcal{P},a)$ and $(\mathcal{P}, b)$ are both VSFT, by the
maximality of $\mathcal{P}$. Hence there exist
finitely generated ideals $\mathcal{B}_a \subseteq (\mathcal{P},a)$
and $\mathcal{B}_b \subseteq (\mathcal{P},b)$ and fixed positive
integers $M,N$ such that $(\mathcal{P},a)^M \subseteq \mathcal{B}_a$
and $(\mathcal{P},b)^N \subseteq \mathcal{B}_b$. 

As $\mathcal{B}_a\subseteq(\mathcal{P},a)$ we can assume that $\mathcal{B}_a=(x_1,x_2,\cdots ,x_s, a)$ with each $x_i\in\mathcal{P}$. Similarly, we write $\mathcal{B}_b=(y_1,y_2,\cdots ,y_t, b)$ with each $y_i\in\mathcal{P}$. If we define $\mathcal{B}:=\mathcal{B}_a\mathcal{B}_b$, then

\[
\mathcal{B}=(\{x_iy_j\}_{i,j}, \{ay_j\}_j, \{bx_i\}_i, ab).
\]

Note that the generator $ab\in\mathcal{P}$ by assumption and the generators of the form $x_iy_j, ay_j$ and $bx_i$ are in $\mathcal{P}$ since each $x_i, y_j\in\mathcal{P}$. Collecting our observations, we note that $\mathcal{B}\subseteq\mathcal{P}$ and is finitely generated. Also since $\mathcal{P}^M\subseteq \mathcal{B}_a$ and $\mathcal{P}^N\subseteq\mathcal{B}_b$, then $\mathcal{P}^{M+N}\subseteq\mathcal{B}$, which is a contradiction, as we have established that $\mathcal{P}$ is not VSFT. We conclude that $\mathcal{P}$ is prime.

We see that if $R$ has an ideal that is not VSFT, then $R$ must have a prime ideal that is not VSFT. This completes the proof.
\end{proof}

If $I$ is a (V)SFT ideal with data $\{I,\mathcal{B}, n\}$, it is useful to know how the data in the second two components can be varied. We will see that by varying the third item of data, we can allow the finitely generated sub-ideal of $I$ to be any finitely generated ideal contained in $I$ if its radical contains $I$. We formalize this in the next result which will prove especially useful in the verification of examples.

\begin{thm}\label{anyradical}
Let $I$ be a VSFT (resp. SFT) ideal. If $\mathcal{B}$ is any finitely generated ideal with $\mathcal{B}\subseteq I\subseteq\sqrt{\mathcal{B}}$ then there is a positive integer $m$ such that $I$ has VSFT (resp. SFT) data $(I,\mathcal{B}, m)$
\end{thm}

\begin{proof}
Let $I$ be VSFT and $\mathcal{B}$ a finitely generated ideal such that $\mathcal{B}\subseteq I\subseteq\sqrt{\mathcal{B}}$. Since $I$ is VSFT, there is a finitely generated ideal $J\subseteq I$ and a positive integer $n$ such that $I^n\subseteq J$. To show that there is a positive integer $m$ such that $I^m\subseteq\mathcal{B}$, we 
pass to the quotient $R/\mathcal{B}$. The image of $J$ in $R/\mathcal{B}$ is finitely generated, and since $J\subseteq I\subseteq\sqrt{\mathcal{B}}$, the image of $J$ is a nil ideal in $R/\mathcal{B}$. Since any finitely generated nil ideal is nilpotent, we must have that there is a positive integer $k$ such that $J^k\subseteq\mathcal{B}$. Hence $I^{nk}\subseteq J^k\subseteq\mathcal{B}$. The proof for the statement for the SFT case is almost identical.
\end{proof}

Finally, we present a theorem that shows that the VSFT property is rather encompassing in the sense that if $J$ is a radical ideal with the VSFT property and $\sqrt{I}=J$ then $I$ is also VSFT.

\begin{thm}\label{modifiedRadicalPower}
Let $R$ be a ring, and $I$ an ideal of $R$ such that its radical, $J$, is a VSFT ideal with data $(J,\mathcal{B},n)$. Then there exists a positive integer $k > 0$ such that $J^k \subseteq I$. Moreover, $I$ is also a VSFT ideal.
\end{thm}

\begin{proof}
We prove the second statement (that $I$ is VSFT) first.  Since $\sqrt{I}=J$, $J/I$ is a nil ideal of $R/I$ and since $\mathcal{B}\subseteq J$ is finitely generated and $\sqrt{B}=J$, the image of $\mathcal{B}$ under the canonical projection $R\longrightarrow R/I$ is a nil (and hence nilpotent) ideal.  Hence there is an $m\in\mathbb{N}$ such that $\mathcal{B}^m\subseteq I$. As $\mathcal{B}^m\subseteq I$ and $J^{nm}\subseteq\mathcal{B}^m$, we have $I^{nm}\subseteq J^{nm}\subseteq\mathcal{B}^m$ and so $I$ is VSFT with data $(I,\mathcal{B}^m, nm)$.

For the first statement, we refer to the above and note that $J^{nm}\subseteq \mathcal{B}^m\subseteq I$.
\end{proof}

\section{Bifurcation of the Properties}

With the results of the previous section at our disposal, we are now in a better position to
construct examples of rings which are SFT but not VSFT.
The first example that we provide is an example of a 1-dimensional quasilocal
domain which is an integral extension of a 1-dimensional Noetherian
valuation domain that is SFT but not VSFT.

\begin{ex}\label{SFTnotVSFTex3FinChar}

We consider the domain 

\[
R=\mathbb{F}_2(y_1^2,y_2^2,y_3^2,\cdots.)[[x^2]][xy_1,xy_2,xy_3,\cdots]
\]

The ideal $I = (x^2,xy_1,xy_2,xy_3,\cdots)$ is the unique maximal ideal of
$R$ (it is easy to see that if $z\notin I$ then $z$ is a unit). Since $R$ is an integral extension of a power series ring over a field, $I$ is the unique nonzero prime ideal of $R$. We consider the
finitely generated ideal $\mathcal{B}= (x^2)\subseteq I$ and note that $I=\sqrt{\mathcal{B}}$. We observe that $(xy_n)^2 = (x^2)(y_n^2) \in (x^2) $
for all $ n \geq 1$, and since $R$ is of characteristic 2, then this extends to sums of the generators of $I$. On the other hand, if $k_1,k_2,\cdots ,k_m$ are distinct positive integers, then $(xy_{k_1})(xy_{k_2})\cdots (xy_{k_m})\notin (x^2)$.
This shows that $I$ is SFT, but Theorem \ref{anyradical} allows us to conclude that $I$ is not VSFT. Since $I$ is the only nonzero prime ideal
of $R$, by Theorem \ref{PrimeSufficient}, $R$ is not VSFT, but is an SFT domain. 

\end{ex}

In the next example, we provide an example of an SFT domain which is not VSFT, in characteristic zero.

\begin{ex}\label{SFTnotVSFTex4Char0}

We begin with the domain: 

\[
T =
\mathbb{Z}_{(2)}[2^{1+\frac{1}{2}},
2^{2+\frac{1}{2^2}},2^{3+\frac{1}{2^3}},\cdots ,2^{n+\frac{1}{2^n}},\cdots]
\]

\noindent and use this to define $R:=T_{\mathfrak{N}}$
where $\mathfrak{N}$ is the maximal ideal of $T$ generated by the set $\{2, 2^{1+\frac{1}{2}},
2^{2+\frac{1}{2^2}},2^{3+\frac{1}{2^3}},\cdots ,2^{n+\frac{1}{2^n}},\cdots\}$. Also note that for all $n$, the element $2^{n+\frac{1}{2^n}}$ is
integral over $\mathbb{Z}_{(2)}$ since $(2^{n+\frac{1}{2^n}})^{2^n}=
2^{n2^n + 1} \in \mathbb{Z}_{(2)}$, and so, since  $\mathbb{Z}_{(2)}$ is 1 dimensional and $R$ is integral over $\mathbb{Z}_{(2)}$, $R$ is also 1-dimensional. It is also easy to see that $R$ is quasilocal (and is of characteristic 0). 

In the computation below, we show that $\mathfrak{M}:=R\mathfrak{N}$ is SFT with
the data $(\mathfrak{M},(2), 2)$, but $\mathfrak{M}$ is not VSFT.

First note that for all $n\geq 1$, $(2^{n+\frac{1}{2^n}})^2 =
(2)(2^n)(2^{(n-1)+\frac{1}{2^{(n-1)}}}) \in (2)$; from this it follows easily from the multinomial formula that the square of any element of $\mathfrak{M}$ is an element of $(2)$. Hence we have established that $\mathfrak{M}$ is SFT with data $(\frak{M}, (2), 2)$.

With regard to the VSFT property, we note that Theorem \ref{anyradical} shows that if $\mathfrak{M}$ is VSFT, then it is VSFT with data $(\mathfrak{M}, (2), k)$ for some positive integer $k$. Consider
the product of $k$ distinct generators:
\[ (2^{n_1+\frac{1}{2^{n_1}}})(2^{n_2+\frac{1}{2^{n_2}}})(2^{n_3+\frac{1}{2^{n_3}}})\cdots(2^{n_k+\frac{1}{2^{n_k}}})\]
which we rewrite as
\[ 2^{n_1+n_2+\cdots+n_k + \frac{1}{2^{n_1}} + \frac{1}{2^{n_2}}+\cdots+\frac{1}{2^{n_k}}},\]
where, without loss of generality, we assume that $n_1 > n_2 >
\cdots > n_k$. 
By way of contradiction, we assume that the product above is an element of $(2)$. If this is the case, then we can write

\[(2^{n_1+\frac{1}{2^{n_1}}})(2^{n_2+\frac{1}{2^{n_2}}})(2^{n_3+\frac{1}{2^{n_3}}})\cdots(2^{n_k+\frac{1}{2^{n_k}}})\]

\[= (2)(2^{n_1+\cdots+n_k+\frac{1}{2^{n_1}}+\cdots+\frac{1}{2^{n_k}}-1}).\]

For the element
\[2^{n_1+\cdots+n_k+\frac{1}{2^{n_1}}+\cdots+\frac{1}{2^{n_k}}-1}\] to be
in $R$, it must be the case that

\[
2^{(n_1+\cdots+n_k+\frac{1}{2^{n_1}}+\cdots+\frac{1}{2^{n_k}}-1)} =
2^{(n'_1+\cdots+n'_s+\frac{1}{2^{n'_1}}+\cdots\frac{1}{2^{n'_s}})}2^M,
\] 

for integers $M$ and $n'_1\geq n'_2\geq\cdots\geq n'_s$. This in turn can be written as 

\[
2^{(n_1+\cdots+n_k+\frac{1}{2^{n_1}}+\cdots+\frac{1}{2^{n_k}}-1)} =
2^{(\frac{1}{2^{m_1}}+\cdots\frac{1}{2^{m_t}})}2^N,
\] 

where $N$ is a positive integer and $m_1>m_2>\cdots >m_t$, and note that $N\geq m_1+m_2+\cdots +m_t$.

Comparing the exponents of $2$, we see that

\[
(\frac{1}{2^{n_1}}+\cdots+\frac{1}{2^{n_k}})-(\frac{1}{2^{m_1}}+\cdots+\frac{1}{2^{m_t}})\in\mathbb{Z}.
\]

Note that $n_1=m_1$ and by induction, we obtain that $k=t$ and $n_i=m_i$ for all $1\leq i\leq k$. From the above, we obtain that 

\[
2^{(n_1+\cdots +n_k-1)}=2^N=2^{(m_1+\cdots +m_t-1)}
\]

which is a contradiction since $N\geq m_1+m_2+\cdots +m_t$.

We have shown that for all $k$ we can find a product of $k$ elements of the radical of $(2)$ that is not in $(2)$. By Theorem \ref{anyradical}, $R$ is not VSFT.

\end{ex}

\section{Flatness and Behavior in Extensions}

Integral extensions are often thought of as ``nice" extensions as many fundamental properties of ring theoretic importance are preserved in integral extensions. However, we will see in the next example that the (V)SFT property does not necessarily behave well in integral extensions. In particular, we will produce an example where the integral closure of a VSFT ring is not even a
SFT ring. It is worth noting that, in general, integral extensions (or even integral closures) of Noetherian rings need not be Noetherian (although the examples are highly nontrivial, the interested reader should see \cite{Nagata1954} for example).

The following example was used in a similar context in~\cite{Co2002}.

\begin{ex}\label{IntClosureNotVSFT}
Let us consider a non-discrete valuation domain $V$ with
value group $\mathbb{Q}$; specifically, we let $A
=\mathbb{F}_2[x;\mathbb{Q}^+_0]$ be the monoid domain with nonnegative rational exponents over the field of $2$ elements, and $V=A_{\mathfrak{M}}$ where $\mathfrak{M}$ is the
maximal ideal of $A$, given by $\mathfrak{M}=
(\{x^{\alpha}\}_{\alpha > 0})$. Let $R:= \mathbb{F}_2 + xV$. $R$ has a unique nonzero prime ideal given by $xV$ and it is easy to see that this ideal is VSFT with data $(xV, (x), 2)$. By Theorem
\ref{PrimeSufficient}, $R$ is a 1-dimensional VSFT domain. However, the
domain $V$ is the integral closure of $R$ and is not even an SFT
domain and, hence, cannot be VSFT.
\end{ex}

This example also reveals the fact that an overring of a VSFT domain
may not even be an SFT domain. However, as the next result shows, under
reasonable restriction, the VSFT property can be extended to an
overring of the domain. Before we proceed to state and prove the
result, we recall the definition of a flat module (extension) and the notion of generalized transforms. Details can be found in
~\cite{ab}.

\begin{defn}
Let $A$ be an R-module. We say that $A$ is a flat $R-$module if the functor $-\otimes_RA$ is exact. Additionally, we say that a ring extension $T$ of $R$ is flat if $T$ is flat as an $R-$module.
\end{defn}

Let $R$ be a domain with quotient field $K$. Following \cite{ab}, we let $\mathfrak{F}$ be a
multiplicatively closed set of ideals in $R$. Let us define
$R_{\mathfrak{F}}$ as the set $\{z \in K | zA \subseteq R$ for some
$ A \in \mathfrak{F}\}$. Moreover, if $T$ is a flat overring of $R$,
then we can choose $\mathfrak{F}$ such that $T = R_{\mathfrak{F}}$ and $AT = T$ for all $A \in
\mathfrak{F}$.

Armed with this tool, we prove the next result.

\begin{thm}\label{flatOverring}
Let $R$ be a VSFT domain. If $T$ is a flat overring of $R$, then $T$
is also VSFT.
\end{thm}

\begin{proof}
We shall show that every prime ideal of the flat
overring $T$ must be VSFT and then by Theorem
\ref{PrimeSufficient} the result will follow.

Let $Q$ be a prime ideal of $T$. Let $P = Q \bigcap R $. By
hypothesis, $P$ is VSFT. So there exist a finitely generated ideal
$\mathcal{B} \subseteq P$ and a positive integer $N$ such that $P^N
\subseteq \mathcal{B}$. Now, by a result proved in
~\cite{ab} (see the proof of Theorem 1.2), we can find a multiplicatively
closed set of ideals $\mathfrak{F}$ in $R$ such that $Q =
P_{\mathfrak{F}}$. Therefore, for $q \in Q$, $qA \subseteq P$ for
some $A \in \mathfrak{F}$. Let $q_1,q_2,\cdots,q_N$ be $N$ arbitrary
elements of $Q$. So there exist $N$ ideals $A_1,A_2,\cdots,A_N$ in
$\mathfrak{F}$ such that $q_1A_1 \subseteq P, q_2A_2 \subseteq
P,\cdots,q_NA_N \subseteq P$. Therefore,
$(q_1q_2\cdots q_N)(A_1A_2\cdots A_N) \subseteq P^N \subseteq \mathcal{B}$.

Since $\mathfrak{F}$ is multiplicatively closed, $A = A_1A_2\cdots A_N
\in \mathfrak{F}$. So, $(q_1q_2\cdots q_N)A \subseteq \mathcal{B}$. Let
us now define a subset $\mathfrak{A}$ of $T$ by $\{z \in T
|(q_1q_2\cdots q_N)z \in \mathcal{B}T\}$ and note that $\mathfrak{A}$ is an
ideal of $T$. 

At this point, we note that $A \subseteq \mathfrak{A}$. Since $A \in
\mathfrak{F}$ and we know that $\mathfrak{F}$ can be so chosen that
$AT = T$ for $A \in \mathfrak{F}$, we conclude that $\mathfrak{A}=
T$ (note that $T = AT \subseteq \mathfrak{A}T = \mathfrak{A}$). Since
$1 \in T = \mathfrak{A}$, $(q_1q_2\cdots q_N)1 \in
\mathcal{B}T$. Note that, $\mathcal{B}T$ is a finitely generated
ideal of $T$ such that $\mathcal{B}T \subseteq Q$ and, any
arbitrary product of $N$ elements from $Q$ is in $\mathcal{B}T$.
Hence, $Q^N \subseteq \mathcal{B}T$ and therefore, $Q$ is VSFT.
\end{proof}

An immediate corollary of the above theorem can be found by
observing the fact that a localization of a domain is a flat
overring of the domain.

\begin{cor}
Let $R$ be a domain and $S$ is a multiplicative set in $R$. If $R$
is VSFT, then so is $R_S$.
\end{cor}

\begin{cor}\label{Pcor}
Let $R$ be a VSFT Pr$\ddot{\text{u}}$fer domain, then every overring of $R$
is VSFT.
\end{cor}

\begin{proof}
It is well-known that every overring of a
Pr$\ddot{\text{u}}$fer domain is a flat overring. Hence by Theorem
\ref{flatOverring}, we know that every overring is also VSFT.
\end{proof}

It is worth noting that Pr\"{u}fer domains need not be VSFT, but every overring of a Pr\"{u}fer domain is a Pr\"{u}fer domain. So suppose one considers a chain of distinct Pr\"{u}fer domains $\{D_\alpha\}_{\alpha\in\Lambda}$ all contained in the same quotient field, and we totally order the indexing set $\Lambda$ by declaring that $\alpha\leq \beta$ if and only if $D_\alpha\subseteq D_\beta$. Corollary \ref{Pcor} shows that if there is an $\alpha\in\Lambda$ such that $D_\alpha$ is VSFT, then $D_\beta$ is VSFT for all $\beta\in\Lambda$ such that $\alpha\leq \beta$.

We also record VSFT behavior in homomorphic images. As is the case for the Noetherian property and the SFT property, the VSFT property is preserved in  homomorphic images.

\begin{thm}
The homomorphic image of a VSFT ring is VSFT.
\end{thm}

\begin{proof}
Let $R$ be a VSFT ring and $S$ be its homomorphic
image under the homomorphism $\phi$. By Theorem
\ref{PrimeSufficient}, we need to show that every prime ideal in $S$
is VSFT.

Let $P$ be a prime ideal of $S$. Then it is well known fact that
$\phi^{-1}(P)$ is a prime ideal of $R$. By hypothesis,
$\phi^{-1}(P)$ is a VSFT ideal. So there exists a finitely generated
ideal $\mathcal{B}\subseteq\phi^{-1}(P)$ of $R$ and a positive integer $N$ such that
$(\phi^{-1}(P))^N \subseteq \mathcal{B}$. Let generators of
$\mathcal{B}$ be given by $r_1,r_2,\cdots,r_k$. Now consider $N$
arbitrary elements $p_1, p_2,\cdots,p_N $ from $P$. There exists $q_i
\in \phi^{-1}(P)$ such that $p_i = \phi(q_i), $ for all $ i$. From
the VSFT data of $\phi^{-1}(P)$, we conclude that $q_1q_2\cdots q_N \in
\mathcal{B}$. Hence, $q_1q_2\cdots q_N = s_1r_1 + s_2r_2 + \cdots +
s_kr_k$ for $s_i \in R, $ for all $ i$. Therefore,
$\phi(q_1q_2\cdots q_N) =\phi( s_1r_1 + s_2r_2 + \cdots + s_kr_k)$ or
$p_1p_2\cdots p_N =
\phi(s_1)\phi(r_1)+\phi(s_2)\phi(r_2)+\cdots+\phi(s_k)\phi(r_k)$. So
there exists a finitely generated ideal $\mathcal{A}:=\phi(\mathcal{B})= (\phi(r_1),\phi(r_2),\cdots,\phi(r_k))$ in $S$, which
is contained in $P$ such that $P^N \subseteq \mathcal{A}$. Hence, $P$ is
VSFT, and this completes the proof. 
\end{proof}

For the sake of completeness we record some useful properties of VSFT rings. These are all properties enjoyed by SFT rings (see \cite{ar1973b}), so we list these properties without proof.

\begin{thm}
If $R$ is an SFT ring, then $R$ has the following properties.
\begin{enumerate}
\item $R$ satisfies the ascending chain condition on radical ideals.
\item Any ideal $I\subseteq R$ has only finitely many primes minimal over it.
\item Any radical ideal of $R$ is the intersection of finitely many prime ideals.
\item Any radical ideal is the radical of a finitely generated ideal.
\end{enumerate}
\end{thm}

We now consider the extension of (V)SFT ideals to ring extensions. In particular, the next theorem shows that if a VSFT ideal survives in an extension ring, then the
extended ideal also satisfies the VSFT property. An analogous result is true for SFT rings, but the proof is more subtle and will be presented separately.

\begin{thm}\label{vsftsurvives}
Let $R$ be a ring contained in another ring $T$. Let $I$ be a VSFT
ideal of $R$ with data $(I, \mathcal{B}, N)$, which survives in
$T$. Then $IT$ is also VSFT with data $(IT, \mathcal{B}T, N)$.
\end{thm}

\begin{proof} 
By hypothesis $\mathcal{B}$ is finitely generated and
so is the extension $\mathcal{B}T$. Also, we note that
$\mathcal{B}T \subseteq IT$. A typical element of $IT$ looks like
$i_1t_1 + i_2t_2 + i_3t_3 + \cdots + i_nt_n$, where $i_m \in I$ and $t_m \in T, $ for all $ m$. Let $s_1, s_2,\cdots,s_N$ be
$N$ arbitrary elements from $IT$. Then $s_1 = i_{11}t_{11} +
i_{12}t_{12} + \cdots + i_{1n_1}t_{1n_1}$, $s_2 = i_{21}t_{21} +
i_{22}t_{22} + \cdots + i_{2n_2}t_{2n_2},\cdots, s_N = i_{N1}t_{N1} +
i_{N2}t_{N2} + \cdots + i_{Nn_N}t_{Nn_N}$. A typical term of the
product $s_1s_2\cdots s_N$ is given by
$i_{1k_1}i_{2k_2}\cdots i_{Nk_N}t_{1k_1}t_{2k_2}\cdots t_{Nk_N}$, where
$k_1,k_2,\cdots,k_N$ are appropriate integers denoting the terms from
each element and $j \leq k_j \leq n_j, $ for all $ j$ such that $1
\leq j \leq N$. By hypothesis, $I$ is VSFT with data
$(I,\mathcal{B}, N)$. So $I^N \subseteq \mathcal{B}$, and hence,
$i_{1k_1}i_{2k_2}\cdots i_{Nk_N} \in \mathcal{B}$ which implies
$i_{1k_1}t_{1k_1}i_{2k_2}t_{2k_2}\cdots i_{Nk_N}t_{Nk_N} \in
\mathcal{B}T$. Each term in the product is in $\mathcal{B}T$ and
so is their sum. Hence, $s_1s_2\cdots s_N \in \mathcal{B}T$. Thus
$(IT)^N \subseteq \mathcal{B}T$ and hence $IT$ is VSFT.
\end{proof}

A tool that will be useful for the SFT case is the following theorem (Theorem 2.1 from ~\cite{coco}).

\begin{thm}\label{StrongConvergence}
Let $\Gamma$ be an SFT ideal of a ring $R$ with data $(\Gamma,
B, N)$. If $\{a_i\}_{i=1}^N$ is a collection of elements of
$\Gamma$, then
\[\sum_{k_1+k_2+\cdots +k_m=N}(\frac{N!}{k_1!k_2!\cdots k_m!})a_1^{k_1}a_2^{k_2}\cdots a_m^{k_m} \in B\]

where $n_j \geq 1, $ for all $ j$. In particular, $N!a_1a_2\cdots a_N
\in B$.
\end{thm}

To proceed, we need an additional technical result. In the sequel, the notation $\lfloor x\rfloor$ denotes the floor function. If $p$ is a positive prime integer, we define $v_p:\mathbb{N}\longrightarrow\mathbb{N}_0$ to be the $p-$adic valuation, and we define $f_p:\mathbb{R}^+\longrightarrow\mathbb{Z}$ by

\[
f_p(m)=\sum_{k\geq 1}\lfloor \frac{m}{p^k}\rfloor.
\]

The utility of the functions $f_p$ and $v_p$ lies in the following well-known result that can be found in a number of places in the literature (for example \cite{MR242767}).

\begin{prop}\label{legendre}
\emph{(Legendre)}
For all $n\in\mathbb{N}$, $v_p(n!)=\sum_{k\geq 1}\lfloor\frac{n}{p^k}\rfloor=f_p(n)$.
\end{prop}

Additionally, we introduce the following proposition designed to deal with our SFT extension theorem.

\begin{prop}\label{floor}
For all real numbers $N> M\geq a_1\geq a_2\geq \cdots\geq a_m>0$ such that

\[
a_1+a_2+\cdots +a_m\leq NM,
\] 

\noindent we have

\[
f_p(NM)\geq f_p(N)+f_p(a_1)+f_p(a_2)+\cdots +f_p(a_m).
\] 
\end{prop}

\begin{proof}
Note that there is a nonnegative integer $k$ such that $p^k\leq M<p^{k+1}$; we proceed by induction on $k$. 

For the base case, we let $k=0$ and note that this means that $f_p(a_i)=0$ for all $i$. Hence we have 

\[
f_p(NM)\geq f_p(N)\geq f_p(N)+f_p(a_1)+f_p(a_2)+\cdots +f_p(a_m),
\]

\noindent and this establishes the base case.

Inductively, we assume that for all $k\leq n$, if $p^k\leq M<p^{k+1}$ and 

\[
a_1+a_2+\cdots +a_m\leq NM
\]

\noindent then we have that 

\[
f_p(NM)\geq f_p(N)+f_p(a_1)+f_p(a_2)+\cdots +f_p(a_m).
\]

Now suppose that $p^{n+1}\leq M<p^{n+2}$ and $a_1+a_2+\cdots +a_m\leq NM$. Note that we have $\frac{a_1}{p}+\frac{a_2}{p}+\cdots +\frac{a_m}{p}\leq \frac{NM}{p}$. Hence

\[
\lfloor\frac{NM}{p}\rfloor\geq \lfloor \frac{a_1}{p}+\frac{a_2}{p}+\cdots +\frac{a_m}{p} \rfloor\geq \lfloor\frac{a_1}{p}\rfloor+\lfloor\frac{a_2}{p}\rfloor+\cdots +\lfloor\frac{a_m}{p}\rfloor
\]

We also observe that $f_p(x)=\lfloor\frac{x}{p}\rfloor +f_p(\frac{x}{p})$.

Since $p^n\leq\frac{M}{p}< p^{n+1}$, we apply the inductive hypothesis to obtain

\[
f_p(NM)=\lfloor\frac{NM}{p}\rfloor+f_p(\frac{NM}{p})\geq\sum_{i=1}^m\lfloor\frac{a_i}{p}\rfloor+f_p(N)+\sum_{i=1}^m f_p(\frac{a_i}{p})\geq
f_p(N)+\sum_{i=1}^mf_p(a_i),
\]

\noindent and this completes the proof.
\end{proof}

\begin{cor}\label{ala}
If $k_1+k_2+\cdots +k_m=NM$ and each $k_i<N$ then $N!$ divides $\frac{(NM)!}{k_1!k_2!\cdots k_m!}$.
\end{cor}

\begin{proof}
It suffices to show that for all positive primes $p, v_p(N!)\leq v_p(\frac{(NM)!}{k_1!k_2!\cdots k_m!})$; in light of Proposition \ref{legendre} this is tantamount to showing that $f_p(NM)\geq f_p(N)+f_p(k_1)+f_p(k_2)+\cdots +f_p(k_m)$, and this follows immediately from Proposition \ref{floor}.
\end{proof}

We now give the analogous result for SFT rings. In this theorem, we will assume that $\Gamma$ is a nontrivial (that is, not finitely generated) SFT ideal. In the case that $\Gamma$ is finitely generated, its SFT data can be chosen as $(\Gamma,\Gamma, 1)$.

\begin{thm}\label{SFText}
Let $R\subseteq T$ be rings. If $\Gamma$ is a nontrivial SFT ideal of $R$ with data $(\Gamma, B, N)$,
then $\Gamma T$ is also an SFT ideal of $T$ with data $(\Gamma
T, BT, N(N-1))$.
\end{thm}

\begin{proof}
Let $\gamma:=\sum_{i=1}^ma_it_i$ be an element of $\Gamma T$ (with $a_i\in \Gamma$ and $t_i\in T$). Note that
\[
\gamma^{N^2-N}=\sum_{k_1+k_2+\cdots +k_m=N^2-N} \frac{(N^2-N)!}{k_1!k_2!\cdots k_m!}(a_1t_1)^{k_1}(a_2t_2)^{k_2}\cdots (a_mt_m)^{k_m}
\]

\noindent where each $k_i\geq 0$.

Of course, if any of the powers, $k_i$, is greater than or equal to $N$, then its corresponding term from the above sum is in $BT$, and so we need only consider the terms such that each $k_i\leq N-1$. 

Under this assumption, let $\frac{(N(N-1))!}{k_1!k_2!\cdots k_m!}(a_1t_1)^{k_1}(a_2t_2)^{k_2}\cdots(a_mt_m)^{k_m}$ be a typical element of the sum above with each $k_i\leq N-1$. As each $k_i\leq N-1$ and $N(N-1)=k_1+k_2+\cdots+k_m$ then $m\geq N$. Since each $a_i\in\Gamma, m\geq N$, and $\frac{(N(N-1))!}{k_1!k_2!\cdots k_m!}$ is divisible by $N!$ (by Corollary \ref{ala}), we have by Theorem \ref{StrongConvergence} that the term $\frac{(N(N-1))!}{k_1!k_2!\cdots k_m!}(a_1t_1)^{k_1}(a_2t_2)^{k_2}\cdots(a_mt_m)^{k_m}\in BT$ and this establishes the theorem.
\end{proof}

We remark that to see that the exponent $N^2-N$, is necessary in some situations, it is easiest to consider a ring of characteristic $0$ containing an SFT ideal $I$ with data $(I,B, N)$ and then consider the ideal $I[x]$ in $R[x]$. Theorem \ref{SFText} shows that if $I$ is SFT, then so is $I[x]$, but we point out that the question as to whether the polynomial extension of an SFT ring is also SFT is still open.

\section{Valuation and Pr\"{u}fer Domains}

In the spirit of the previous results we look at some cases where the SFT and VSFT properties are equivalent and in particular, we explore the case in which our domain is a valuation domain or a Pr\"{u}fer domain. Of course these classes of domains are of great independent interest, but are especially relevant with regard to the applications of the SFT property, in particular, in the study of $\text{Spec}(R[[x]])$ (for example \cite{ar1973b}, \cite{kp2}, \cite{CKT}).

We begin with a result that first appeared in \cite{coco}. This result has a nice corollary that shows in domains that containing fields of characteristic $0$, the VSFT and SFT property are equivalent.

\begin{thm}
Let $R$ be a domain of characteristic $0$ and $\Gamma\subset R$ be a SFT ideal of $R$ with data $(\Gamma, B, N)$. Suppose that $P$ is a prime ideal minimal
over $\Gamma$ such that $P \bigcap \mathbb{Z}= (0)$. Then in the
localization $R_P$, $\Gamma$ survives and $(\Gamma R_P)^N \subseteq
BR_P$. In particular, if $R$ contains the rational numbers, then all
SFT ideals are VSFT.
\end{thm}

We will now show that in valuation domains, and more generally, in
Pr$\ddot{\text{u}}$fer domains, the conditions are equivalent. We begin with a more general statement that is more sweeping in nature.

\begin{thm}\label{rcd}
Let $R$ be a root closed domain and $I\subseteq R$ an SFT ideal with data $(I,\mathcal{B},n)$. If $\mathcal{B}$ can be chosen to be principal, then $I$ is VSFT.
\end{thm}

\begin{proof} 
Let $I$ be an ideal of $R$ with SFT
data $(I, \mathcal{B}, n)$. By hypothesis, we may choose $\mathcal{B} = (\gamma)$ where $\gamma \in R$.
We now choose $n$ arbitrary elements $x_1, x_2, \cdots, x_n$
from $I$. Then for all $i$, $x_i^n = \gamma r_i$, where $r_i \in
R$, and this leads to the equation

\[
x_1^nx_2^n\cdots x_n^n=\gamma^nr_1r_2\cdots r_n.
\]

 Note that
$(\frac{x_1x_2\cdots x_n}{\gamma})^n = r_1r_2\cdots r_n \in R$ and so
as $R$ is root closed, $\frac{x_1x_2\cdots x_n}{\gamma} \in R$. Hence
$x_1x_2\cdots x_n \in (\gamma)$ and $I$ is VSFT.
\end{proof}

The next corollary is immediate from Theorem \ref{rcd} and Theorem \ref{anyradical}.

\begin{cor}\label{rcf}
If $R$ is a root closed domain with the property that any prime ideal is the radical of a principal ideal, then the SFT and VSFT properties are equivalent.
\end{cor}

\begin{proof}
It suffices to show that if $R$ is SFT, then it is VSFT. To this end, we recall that by Theorem \ref{PrimeSufficient} it suffices to show that each prime ideal $P$ is VFST. Suppose that $P\subset R$ is prime with SFT data $(P,\mathcal{B},n)$. Our assumption that $P$ is the radical of a principal ideal, coupled with Theorem \ref{anyradical}, assures that we can select $\mathcal{B}$ to be principal, and now Theorem \ref{rcd} shows that $P$ is VSFT.
\end{proof}

\begin{prop}\label{VDSFT}
In a valuation domain $V$, the SFT and the VSFT conditions are
equivalent.
\end{prop}

\begin{proof} 
If $V$ is a valuation domain, it is integrally closed, and hence root closed. We also observe that if $V$ is SFT, then every prime ideal is the radical of a finitely generated, and hence principal, ideal as $V$ is a valuation domain. Corollary \ref{rcf} applies.
\end{proof}

\begin{prop}
In a Pr$\ddot{\text{u}}$fer domain, the SFT and the VSFT conditions
are equivalent.
\end{prop}

\begin{proof} 
The proof of this fact depends heavily on Proposition 3.1 from \cite{ar1973b}. This result states that a prime ideal $P$ of a Pr\"{u}fer domain is SFT if
and only if there exists a finitely generated ideal $A$ in $P$ such
that $P^2 \subset A \subset P$ or in other words, if and only if $P$
is VSFT.
\end{proof}

\begin{rem}
It is clear from the above theorem that if $R$
is integrally closed (or completely integrally closed or
$\Omega$-almost integrally closed, see \cite{cd2}), then the VSFT and SFT conditions
are equivalent, provided the sub-ideal can be chosen principal. Since the primes used in the the examples presented in Section 3 had the property that they were the radical of a principal ideal, the key was that neither of the domains were integrally closed.
\end{rem}

\bibliography{biblio2}{}
\bibliographystyle{plain}

\end{document}